\documentclass[12pt]{amsart}
\usepackage{latexsym, amsmath, amssymb, amsthm, mathrsfs,wasysym}
\usepackage[alwaysadjust]{paralist}
\usepackage{caption}
\usepackage{subcaption}
\usepackage[T1]{fontenc}
\usepackage{lmodern}
\usepackage[backref,colorlinks=true,linkcolor=blue,urlcolor=blue, citecolor=blue]%
  {hyperref}
\usepackage[shortalphabetic]{amsrefs}
\usepackage[all]{xy}
\usepackage[T1]{fontenc}
\usepackage{mathptmx}
\usepackage{microtype}
\usepackage{framed}
\usepackage{tikz}
\usepackage{graphicx}
\usepackage[alwaysadjust]{paralist}
\makeatletter
\providecommand\@enum@widestlabel{7}
\makeatother

\usepackage[centering, includeheadfoot, hmargin=1in, vmargin=1in,
  headheight=30.4pt]{geometry}
\newtheorem{lemma}{Lemma}[section]
\newtheorem{theorem}[lemma]{Theorem}
\newtheorem{corollary}[lemma]{Corollary}
\newtheorem{proposition}[lemma]{Proposition}
\newtheorem{conjecture}[lemma]{Conjecture}
\theoremstyle{definition}

\newtheorem{remark}[lemma]{Remark}

\renewcommand{\theequation}%
{\arabic{section}.\arabic{lemma}.\arabic{equation}}

\newcommand{\CC}{\ensuremath{\mathbb{C}}} 
 
\newcommand{\PP}{\ensuremath{\mathbb{P}}} 
\newcommand{\QQ}{\ensuremath{\mathbb{Q}}} 
\newcommand{\RR}{\ensuremath{\mathbb{R}}} 
 
\newcommand{\sI}{\ensuremath{\kern -1pt \mathscr{I}\kern -2pt}} 
\newcommand{\sJ}{\ensuremath{\kern -2pt \mathscr{J}\kern -2pt}} 

\newcommand{\sO}{\ensuremath{\mathscr{O}}}

\newcommand{\m}{\ensuremath{\mathfrak{m}}}

\renewcommand{\geq}{\geqslant}
\renewcommand{\leq}{\leqslant}

\DeclareMathOperator{\mult}{mult}

\newcommand{\deq}{\ensuremath{\stackrel{\textrm{def}}{=}}}

\newcommand{\Bplus}{\ensuremath{\textbf{\textup{B}}_{+} }}
\newcommand{\Bminus}{\ensuremath{\textbf{\textup{B}}_{-} }}
\newcommand{\Bstable}{\ensuremath{\textbf{\textup{B}} }}

\input xy
\xyoption{all}

\begin{document}

\title{Multiplicities of irreducible theta divisors}

\author[V.~Lozovanu]{Victor Lozovanu}

\address{Victor Lozovanu -- Current Address: Universit\'a degli Studi di Genova, 
\newline  
\hspace*{2.74in} Dipartimento di Matematica, 
\newline  
\hspace*{2.74in} Via Dodecaneso 35, 16146, Genova, Italy.\newline
\hspace*{2.74in} \textit{Email address}: \href{lozovanu@dima.unige.it}{\nolinkurl{lozovanu@dima.unige.it}}}

\address{\hspace*{1.9in} Address: Leibniz Universit\"{a}t Hannover,  \newline
 \hspace*{2.74in} Institut f\"{u}r Algebraische Geometrie, 
 \newline
 \hspace*{2.74in}
 Welfengarten 1, 30167, Hannover, Germany.
 \newline
\hspace*{2.74in} \textit{Email address}: \href{lozovanu@math.uni-hannover.de}{\nolinkurl{lozovanu@math.uni-hannover.de}}}

\maketitle

\begin{abstract}
	Let $(A,\Theta)$ be a complex principally polarized abelian variety of dimension $g\geq 4$. Based on vanishing theorems, differentiation techniques and intersection theory, we show that whenever the theta divisor $\Theta$ is irreducible, its multiplicity at any point is at most $g-2$. This improves work of Koll\'ar \cite{K95}, Smith-Varley \cite{SV96}, and Ein--Lazarsfeld \cite{EL97}. We also introduce some new ideas to study the same type of questions for pluri-theta divisors. 
\end{abstract}

\section{Introduction}
Let $(A,\Theta)$ be a complex $g$-dimensional principally polarized abelian variety (ppav). As a subscheme of $A$, the choice of theta divisor $\Theta$ is in some sense unique and minimal. Thus, one expects very interesting geometric phenomena. For example, studying the singularities of $\Theta$ is a fundamental problem on its own, but it has also deep connections to the Schottky problem, going back to Adreotti--Mayer \cite{AM67}, and can be used to characterize meaningful geometric loci on the moduli space of such pairs. See \cite{C08a} and \cite{GH13} for nice surveys of this circle of ideas.

It was observed by Koll\'ar \cite[Chapter~17]{K95} that vanishing theorems play an important role in understanding the singularities of theta divisors. He shows 
\[
\mult_x(\Theta)\ \leq \ g, \textup{ for all } x\in A.
\]
Smith--Varley \cite{SV96} prove that equality holds if and only if $(A,\Theta)$ is a product of elliptic curves. 

Using the theory of multiplier ideals, Ein--Lazarsfeld \cite{EL97} generalize this result even further. More specifically, they show that for any $k\geq 1$ the subset 
\[
A\ \supseteq \ \Sigma_k(\Theta) \ \deq \ \{x\in A |  \mult_x(\Theta)\geq k\}
\]
has codimension at least $k$. Equality holds if and only if $(A,\Theta)$ is a $k$-fold product of PPAVs. 

These ideas suggest that $\Theta$ has very interesting geometry whenever $(A,\Theta)$ is indecomposable, i.e. $\Theta$ is irreducible by the Decomposition Theorem. For example, as a consequence of \cite{EL97} we know that in this case $\Theta$ is normal with rational singularities, thus proving a conjecture of Arbarello-De Concini \cite{ADC87}. Later, Hacon \cite{H99} generalized these results even further.

Furthermore, one expects better upper bounds on the multiplicity for indecomposable pairs $(A,\Theta)$.  A folklore conjecture, see \cite[Conjecture~29.1]{MP19} or its generalizations \cite{C08a}, states:
\begin{conjecture}\label{conj:main}
	Let $(A,\Theta)$ be an indecomposable ppav of dimension $g\geq 1$. Then
	\[
	\mult_x(\Theta) \ \leq \ \Big\lfloor\frac{g+1}{2}\Big\rfloor\ for \ all \ x\in A \ .
	\]
\end{conjecture}
Besides its obvious aesthetics, Conjecture~\ref{conj:main} and its related forms have important applications to the geometry of the moduli space of ppavs. It is believed that equality holds if and only if $(A,\Theta)$ is the Jacobian of a hyperelliptic curve or the intermediate Jacobian of a smooth cubic threefold. So, certain loci on the moduli, defined by the multiplicity, will then contain only decomposable pairs or certain Jacobians. Moreover, Grushevsky-Hulek  \cite{GH12} use this conjecture to study the locus of intermediate Jacobians of cubic threefolds and its associated Chow class. 

The best current result on Conjecture~\ref{conj:main} is a slightly better bound, when $\Theta$ has isolated singularities, i.e. $\Sigma_2(\Theta)$ is finite, by Musta\c t\u a-Popa \cite[Theorem~29.2(1)]{MP19}, using the theory of Hodge ideals,  and by Codogni-Grushevsky-Sernesi \cite{CGS17}, based on intersection theory and Gau\ss \ map. 

Our first goal was to give a different proof of the aforementioned result, see Corollary~\ref{cor:isolated}. Moreover, our approach proves Conjecture~\ref{conj:main} in a slightly more general setup, when there is a bound on the dimension of $\Sigma_k(\Theta)$ for some $k\geq 2$. More specifically, we show:
\begin{proposition}\label{prop:main}
	Let $(A,\Theta)$ be a $g$-dimensional ppav. If for some positive integer $k\geq 1$, we have  
	\[
	g \ \gg \ \max\{\textup{dim}\big(\Sigma_k(\Theta)\big), k\} .
	\]
	Then Conjecture~\ref{conj:main} holds.
	
\end{proposition}
For the proof, we first note that Conjecture~\ref{conj:main} is a local problem and translate it to a global one, by transfering the data to the blow-up $\pi:\overline{A}\rightarrow A$ of the origin $0\in A$, where $E\simeq \PP^{g-1}$ is the exceptional divisor. In this setup, we consider the divisor classes
\[
\Theta_t \ \deq \ \pi^*(\Theta)-tE \ = \ 
\begin{cases}
\text{nef}, \text{ when } 0\leq t \leq\epsilon(\Theta),\\
\text{pseudo-effective}, \text{ when } \epsilon(\Theta)\leq t \leq\mu(\Theta),\\
\text{not pseudo-effective}, \text{ when } \mu(\Theta) < t ,\\
	\end{cases}
	\]
where $\epsilon(\Theta) \geq 1$ is the Seshadri constant, and $\mu(\Theta)$ the infinitesimal width of $\Theta$. As a side note, \cite[Proposition~3.5]{EL97} and \cite[Corollary~2]{H99} translates in this language as $\mu(\Theta)\leq g$.

This infinitesimal view places the focus on the behaviour of the base loci $\Bstable(\Theta_t)$ for $t\in [\epsilon,\mu]$. Differentiation, as in \cite{N96} or \cite[Lemma~1.3]{N05}, yields that for any subvariety $\overline{S}\subseteq \Bstable(\Theta_t)$ one has
\[
\frac{\partial}{\partial t}\Big(t\rightarrowtail \mult_{\overline{S}}(||\Theta_t||)\Big) \ \geq \ 1 \ .
\]
Consequently, one can relate the loci $\pi(\Bstable(\Theta_t))$ and $\Sigma_k(\Theta)$, for any $t\in\RR_+$ and integer $k\geq 1$.  

Finally, the assumption in Proposition~\ref{prop:main} and the failure of Conjecture~\ref{conj:main} yield strong constraints on the dimension of $\pi(\Bstable(\Theta_t))$. Consequently, the pushforwards of $g$ very carefully chosen divisors on $\overline{A}$, all lying in the class of $\Theta$, will have a zero-dimensional intersection. Counting multiplicities and applying B\'ezout's theorem will then lead to a final contradiction.

Moving forward, it is worth asking whether our techniques lead to new non-trivial upper bounds  on the multiplicity. It turns out that we are able to do so, as long as we know what to do when $1< \epsilon(\Theta)<2$. But these bounds are the subject of a conjecture of Debarre \cite{D04}.
\begin{conjecture}[Debarre]
	\label{conj:main2}
	Let $(A,\Theta)$ be a ppav of dimension $g\geq 4$. If $\epsilon(\Theta)<2$, then either $(A,\Theta)$ is decomposable or is the Jacobian of a hyperelliptic curve.
\end{conjecture}
This can be seen as a numerical version of the classical van Geemen and van der Geer's conjecture \cite{GG86} about the singularity locus of divisors in $|2\Theta|$ of multiplicity at least four at the origin.

Currently, Conjecture \ref{conj:main2} is still open, but in \cite{Loz20} the author generalizes it to any polarization and studies it in small dimensions. These ideas and those from \cite{D94} bring to light some new interesting arithmetic properties of the Seshadri constant $\epsilon(\Theta)$. Combining them with the techniques from Proposition~\ref{prop:main}, and the main results of \cite{EL97} lead us then to the main result of the paper.
\begin{theorem}
	\label{thm:main1}
		Suppose that $(A,\Theta)$ is an indecomposable ppav of dimension $g\geq 3$. Then
		\[
		\mult_x(\Theta) \ \leq \ g-2, \ for \ all \ x\in A \ .
		\]
\end{theorem}
The statement was previously proved for $g\leq 5$ in \cite[Theorem~3]{C08b}, and this latter result is a crucial ingredient in the main proof. 

Moving forward, when $\Sigma_g(\Theta)\neq \varnothing$, by \cite{SV96} the pair $(A,\Theta)$ is a product of polarized elliptic curves. So, an interesting consequence of Theorem~\ref{thm:main1}, by combining it inductively with \cite{EL97}, deals with the next case.
\begin{corollary}
Let $(A,\Theta)$ be a ppav of dimension $g\geq 3$ with $\Sigma_g(\Theta)=\varnothing$.  Then
\begin{enumerate}
\item $\textup{dim}\big(\Sigma_{g-1}(\Theta)\big)=0$ $\Longleftrightarrow$ $(A,\Theta)$ is the product of $g-3$ elliptic curves with the Jacobian of a \\ \hspace*{1.65in} smooth hyperelliptic curve of genus $3$.
\item $\textup{dim}\big(\Sigma_{g-1}(\Theta)\big) = 1$ $\Longleftrightarrow$ $(A,\Theta)$ is the product of $g-2$ elliptic curves with the Jacobian of a \\ \hspace*{1.65in} smooth curve of genus $2$.
\end{enumerate}
\end{corollary}
This corollary provides further evidence that one can use the multiplicity to describe important geometric loci on the moduli spaces of ppavs.

The infinitesimal width $\mu(\Theta)$ bounds above the multiplicity at any point for any effective $\QQ$-divisor in the class of $\Theta$. So, it is natural to ask whether there are stronger results than \cite{EL97} and \cite{H99} for indecomposable ppavs, that can be seen as a numerical counterpart of Conjecture~\ref{conj:main}. Assuming Conjecture~\ref{conj:main2}, we can show the following result:
\begin{theorem}
	\label{thm:main2}
	Let $(A,\Theta)$ be an indecomposable ppav. If Conjecture~\ref{conj:main2} holds, then 
	\[
	\mu(\Theta) \ \leq \ g-\frac{g-1}{g+1}\ .
	\] 
\end{theorem}
We believe that $\mu(\Theta)\leq g-1$ for indecomposable ppavs. But this seems unattainable at the moment. However, as Conjecture~\ref{conj:main2} holds in small dimensions, we use intersection theory to show that this better bound holds whenever $g=2,3,4$, see Theorem \ref{thm:abelian3} and \ref{thm:abelian4}.

The approach to Theorem~\ref{thm:main2} is inspired by \cite{LPP11}. One studies global positivity questions on abelian varieties in terms of the existence of effective divisors with certain singularities. This was used successfully in \cite{KL19} and \cite{Loz18} to study syzygies of abelian varieties in low dimensions.

If the upper bound does not hold, then \cite[Proposition~3.5]{EL97} allows us to find two divisors $D',D''\equiv_{\textup{num}}\Theta$ so that $D'+D''$ has zero-dimensional multiplier ideal and $\mult_0(D'+D'')> g+1$. Standard ideas from \cite[Chapter~9]{PAG} imply that these conditions force the linear systems $|2\Theta|$ to separate any two tangency directions at any point on $A$. This is of course not possible as this system defines a $2:1$ (non-\'etale) morphism from $A$ to its Kummer variety $K(A)\subseteq \PP^{2^g-1}$.

\subsection*{History} It is worth mentioning that some of the ideas used in this article have appeared before. The infinitesimal perspective in this area can be traced back to the work of Beauville and Debarre \cite{BD88}. They translated and studied the conjecture of van Geemen and van der Geer on the blow-up at the origin. The use of differentiation techniques and intersection theory in this area seems to appear first in the work of Nakamaye \cite{N97}, who attributes it to Lazarsfeld. He gives there a different proof to Smith-Varley's results on the multiplicity of theta divisors. Lastly, Nakamaye kickstarted the study of Seshadri constants on abelian varieties in \cite{N96}. 

\subsection*{Acknowledgements} The author is greatly indebted to Klaus Hulek, Matthias Sch\"utt and all the members of IAG at Leibniz Universit\"at for the wonderful last three years in Hannover. Many thanks to V\'ictor Gonzalez-Alonso, whose expertise on abelian varieties was of tremendous help during this project, and Carsten Liese for many encouragements and mathematical support. The author would like to thank S. Casalaina-Martin, C. Ciliberto, G. Farkas, M. Fulger, S. Grushevsky, and M. Popa for many important suggestions and corrections.

\section{Notation and preparations}
Throughout this article we work over the field of complex numbers $\CC$. A pair $(A,\Theta)$ stands for a $g$-dimensional principally polarized abelian variety $A$ and a theta divisor $\Theta$. This Cartier divisor is ample with $h^0(A,\sO_A(\Theta))=1$, which is equivalent by asymptotic Riemann-Roch to $\Theta^g=g!$.

Let $\pi: \overline{A}\rightarrow A$ be the blow-up of $A$ at the origin $0$, where $E\simeq \PP^{g-1}$ is the exceptional divisor. For any $t\in [0,\infty)$ denote by
\[
\Theta_t \ \deq \ \pi^*(\Theta)-tE   \ .
\]
In this infinitesimal setup one can associate two invariants to the class $\Theta$. They do not depend on the choice of the base point due to the group structure on $A$. So we consider them only at the origin.

The first one, introduced by Demailly, is the \textit{Seshadri constant}
\[
\epsilon \ = \ \epsilon(\Theta)\ \deq \ \inf_{0\in C\subseteq X}\frac{(\Theta\cdot C)}{\mult_0(C)} \ = \ \sup\{t>0 \ | \ \Theta_t\textup{ is ample}\}  \ ,
\]
where the infimum is taken over all curves passing through the origin. For more details about this invariant the reader is referred to \cite[Chapter~5]{PAG}.

The second one, which we call the \textit{infinitesimal width} of $\Theta$, is defined to be
\[
\mu \ = \ \mu(\Theta) \ \deq\ \textup{max}\{t>0 \ | \ \Theta_t \textup{ is pseudo-effective}\} \ = \ \textup{max}\{\mult_0(D)\ | \ D\sim_{\QQ}\Theta, D\geq 0\} \ .
\]
Now, for any rational $t\in [0,\mu)$ the class $\Theta_t$ is big, so one can associate the \textit{stable base locus} $\Bstable(\Theta_t)$, i.e. the set of points where all effective $\QQ$-divisor in the class of $\Theta_t$ vanish. But this locus is not a numerical invariant, so \cite{ELMNP06} introduced two other loci that are. Approximation of the stable one, they are called the \textit{restricted base locus} $\Bminus(\Theta_t)$ and \textit{the augmented} one $\Bplus(\Theta_t)$. For our purposes suffices to consider the stable locus, as \cite[Lemma 2.3]{Loz18} states that for any $t>0$ there is $\delta_t>0$ such that
\[
\Bplus(\Theta_{t+\delta}) \ = \ \Bstable(\Theta_{t+\delta})\ = \ \Bminus(\Theta_{t+\delta}) \ , \textup{ for any }0<\delta\leq \delta_t \ .
\]
The following lemma tells us more about the behaviour of these base loci on the blow-up. To not complicate notation further, we work on an abelian variety, but the proof works for any smooth ambient space and any ample line bundle on it.
\begin{lemma}\label{lem:lociseshadri}
	Under the above notation, if there is a $q$-dimensional subvariety $V\subseteq A$ with
	\[
	\epsilon(\Theta) \ = \ \sqrt[q]{\frac{L^{q}\cdot V}{\textup{mult}_x(V)}} \ ,
	\]
	then its proper transform $\overline{V}$ is contained in $\Bstable(\Theta_t)$ for any $t>\epsilon(\Theta)$.
\end{lemma}
\begin{proof}
	This is an application of Nakamaye's theorem on describing the augmented base locus of a big and nef class. In particular, by \cite[Theorem~10.3.5]{PAG} we know that $\overline{V}\subseteq \Bplus(\Theta_{\epsilon})$, where $\overline{V}$ is the proper transform of $V$ on the blow-up $\overline{A}$. Finally, \cite[Lemma 1.3]{Loz18} yields the statement.
\end{proof}

\section{Bounds on singularities loci and Conjecture~\ref{conj:main}}
In this section we state and prove a slightly more explicit version of Proposition~\ref{prop:main}. 
\begin{proposition}\label{prop:asymptotic}
	Let $(A,\Theta)$ be a $g$-dimensional ppav. If there is a positive integer $k\leq \frac{g}{10}+1.5$ with
	\[
	g \ \gg \ \textup{dim}\big(\Sigma_k(\Theta)\big) \ ,
	\]
	then Conjecture~\ref{conj:main} holds.
	\end{proposition}
	\begin{remark}
		Taking a closer look at the details at the end of the proof of Proposition~\ref{prop:asymptotic}, and using the inequality in Lemma \ref{lem:numerical} it is not hard to see that Conjecture \ref{conj:main} holds, whenever
		\[
		 \frac{g}{13\cdot \ln(g)} \ > \ \textup{max}\{k,\textup{dim}\big(\Sigma_k(\Theta)\big)\}.
		\]
		This provides a more exact bound for this conjecture to hold than the the asymptotic version in Proposition \ref{prop:main}.
	\end{remark}
\begin{remark}
In the proof of Proposition~\ref{prop:asymptotic} we do not impose $(A,\Theta)$ to be indecomposable. 
\end{remark}

\begin{proof}
	Denote by $r\deq  \textup{dim}(\Sigma_k(\Theta))$. In the following we will assume that 
	\[
	\mult_0(\Theta) \ \geq  \ m+1, \textup{ where } m \ \deq \lfloor (g+1)/2\rfloor \ ,
	\]
	and the goal would be to get a contradiction. 

 Our first goal is to show that our assumptions yield the following upper bound 
	\begin{equation}\label{eq:oneone}
	\textup{dim}\Big(\pi\big(\Bstable(\Theta_t)\big)\Big) \ \leq \ r, \textup{ for any } t\ <\ m-k+2\ .
	\end{equation}
To prove this we assume the inequality does not hold. In particular, there is a subvariety 
	\[
	\overline{V}\ \subseteq \ \Bstable(\Theta_{t_0}), \textup{ for some } t_0<m-k+2,
	\]
	with $\overline{V}\nsubseteq E$ and $\textup{dim}(\overline{V})\geq r+1$. 
	
	Our ambient space $A$ is abelian, so we can use differential operators to differentiate sections. Based on this, Nakamye hints in \cite{N96} at the following lower bound for asymptotic multiplicity:
	\[
	\mult_{\overline{V}}(||\Theta_t||) \ \geq \ t-t_0 \ , \textup{ for any } t\geq t_0 \ ,
	\]
	where $t_0$ is the starting  point for $\overline{V}$ to show up in $\Bstable(\Theta_t)$. The proof of this result in a more general setup is given by Nakamaye in \cite[Lemma~1.3]{N05}, based on differentiation techniques developed in \cite{ELN94}. For abelian varieties it is explained in \cite[Proposition 4.4]{Loz18}.
	
Going back to the actual inequality, it implies in our setup that
	\[
	\mult_{\overline{V}}(||\Theta_{t}||) \ \geq \ t-t_0\ > \ k-1 \ , \textup{ for any } t\geq m+1 \ .
	\]
Since our initial assumption was that $\mult_0(\Theta)\geq m+1$, then this yields $\mult_{V}(\Theta)> k-1$, where $V\deq \pi(\overline{V})$. In particular, $V\subseteq \Sigma_k(\Theta)$ and this contradicts our assumption that this singularity locus has dimension $r$, as $V$ was assumed to be at least $(r+1)$-dimensional.
	
Now that $(\ref{eq:oneone})$ holds, we use intersection theory to get a contradiction. Basically, we choose carefully $g$ effective divisors in the class of $\Theta$, whose intersection will be zero-dimensional, due to the upper bounds on the dimension of $\Bstable(\Theta_t)$. We do not do the intersection directly on the blow-up $\overline{A}$, since we do not know how to deal with the case when $\Bstable(\Theta_t)$ has small dimensional components outside of $E$ and high-dimensional ones contained in $E$. We neither know if this is possible.
	
	For some $0< \delta\ll 1$, we start by choosing $g-r_0$ very general choices of divisors 
	\[
	\overline{D}_1,\  \ldots ,\  \overline{D}_{g-r_0}\ \equiv \  \Theta_{m-k+2-\delta}, \ \textup{ where }r\geq r_0\ \deq \ \textup{dim}\Big(\pi\big(\Bstable(\Theta_{m-k+2-\delta})\big)\Big),.
	\]
We construct these divisors by normalizing a general choice in a very large power of $\Theta_{m-k+2-\delta}$. Set $D_i=\pi_*(\overline{D}_i)$ for each $i=1,\ldots ,g-r_0$. Due to $(\ref{eq:oneone})$, then the choices of these divisors forces then the scheme theoretical intersection of $D_1,\ldots ,D_{g-r_0}$ to be an effective cycle of codimension $g-r_0$. 
	
	Finally, let $D_{g-r_0+1}, \ldots ,D_g$ be the push-forward of very general choices of divisors from the class $\Theta_{\epsilon-\delta}$, which is ample by the definition of the Seshadri constant $\epsilon=\epsilon(\Theta)$. Using all this data it is then easy to deduce that the intersection of $D_1,\ldots, D_g$ is a zero-dimensional effective scheme. 
	
With this in hand we can apply now B\'ezout's theorem, which yields the following inequality
	\[
	(\Theta^g)\ = \ (D_1\cdot \ldots \cdot D_{g-r_0}\cdot D_{g-r_0+1}\cdot \ldots \cdot D_g) \ \geq  \ \prod_{i=1}^{i=g}\mult_0(D_i) \ .
	\] 
	Taking into account how these divisors were constructed and letting $\delta\rightarrow 0$, this implies
	\[
	(\Theta^g)\ =  \ g! \ \geq \ \Big(\frac{g+1}{2}-k+2\Big)^{g-r_0}\ \cdot \ \epsilon^{r_0} \ .
	\]
\cite{N96} shows that $\epsilon(\Theta)\geq 1$ and $(\ref{eq:oneone})$ implies $r\geq r_0$. So, this inequality yields
	\[
	g! \ \geq \ \Big(\frac{g+1}{2}-k+2\Big)^{g-r} \  
	\]
and this leads to a contradiction due to our initial assumptions and the following estimations. Note first that $g!\sim \sqrt{2\pi g}\big(\frac{g}{e}\big)^g$, whenever $g\gg 0$ with $e$ being the Euler number. So under the asumption that $k\leq \frac{g+15}{10}$ we get 
	\[
\lim_{g\rightarrow\infty}\frac{\big(\frac{g-1}{2}-k+2\big)^{g-r}}{g!}\ = \ 	\lim_{g\rightarrow\infty}\frac{\big(\frac{g-1}{2}-k+2\big)^{g-r}}{\sqrt{2\pi g}\big(\frac{g}{e}\big)^g}  \ \geq  \ \lim_{g\rightarrow\infty}\frac{\big(\frac{g}{2.5})^{g-r}}{\sqrt{2\pi g}\big(\frac{g}{e}\big)^g}\ = \ \lim_{g\rightarrow\infty}\frac{\big(\frac{e}{2.5}\big)^{g}}{\sqrt{2\pi g} \big(\frac{g}{2.5}\big)^r} \ > \ 1 \ ,
	\] 
whenever $g\gg r>0$. This leads to a contradiction and finishes the proof.
\end{proof}

Based on this we recover in a simpler fashion some of the previous statements in the literature, e.g. \cite[Theorem~8.1]{P18}, \cite[Theorem~29.2(1) or Theorem~29.5]{MP19} and \cite[Theorem~1.1]{CGS17}.

\begin{corollary}\label{cor:isolated}
	Let $D\in |n\Theta|$ be an effective reduced divisor, whose support  has an isolated singularity at some point $x\in A$, then  
	\[
	\mult_x(D) \ \leq \ n\cdot \epsilon(\Theta)\ + \ 1 \ \leq \ n\cdot \big(g!\big)^{\frac{1}{g}} +1 \ \sim \ n\cdot \frac{g}{e} +1  \ ,
	\]
	where $e$ is the Euler number.
\end{corollary}
\begin{remark}
Due to Lemma \ref{lem:numerical}, Corollary~\ref{cor:isolated} provides a better than Conjecture~\ref{conj:main} for any $g\geq 7$. When $g=5,6$, these two bounds agree as the multiplicity of $\Theta$ is an integer. 
\end{remark}
\begin{proof}[Proof of Corollary~\ref{cor:isolated}]
	Set $D_x\deq D-x$ and assume the upper bound does not hold, i.e.
	\begin{equation}\label{eq:one}
	\mult_0(D_x)\ > \ n\cdot \epsilon(\Theta)\ + \ 1 \ .
	\end{equation}
	By the definition of the Seshadri constant, there is a subvariety $\overline{V}\subseteq \overline{A}$, not contained in the exceptional divisor, with $\overline{V}\subseteq \Bstable(\Theta_t)$ for any $t>\epsilon$. As before, by \cite[Proposition~4.4]{Loz18}, we have
	\[
	\mult_{\overline{V}}(||\Theta_t||) \ \geq \ t-\epsilon \ , \textup{ for any } t\geq \epsilon \ ,
	\]
	Applying this to $D_x$ and using $(\ref{eq:one})$, we get
	\[
	\mult_V(D_x) \ \geq \  n\cdot \epsilon(\Theta)+1 \ - \ n\cdot \epsilon(\Theta) \ > \ 1 \ .
	\] 
	Hence, $x\in\textup{Supp}(D)$ is not an isolated singularity, leading to a contradiction.
	
	The second inequality follows easily from applying Nakai-Moishezon to the nef class $\Theta_{\epsilon}$. 
\end{proof}
\begin{lemma}\label{lem:numerical}
For any positive integer $g\geq 6$, the following inequality holds:
\[
\Big(\frac{g}{2.5}\Big)^g \ > \ g!.
\]
\end{lemma}
\begin{proof}
Bernoulli's inequality: $(1+x)^g>1+gx$, for any $x>-1$, and $g\geq 1$, implies 
\[
\Big(1 \ - \ \frac{1}{(g+1)^2}\Big)^g \ > \ 1-\frac{g}{(g+1)^2} \ \textup{ for any } g\geq 1 \ .
\]
This forces the sequence $a_g \deq\big(1+\frac{1}{g}\big)^g$ to be increasing and as $a_6>2.5$, then in reality $a_g>2.5 $ for any $g\geq 6$.

Using this lower bound and the induction process lead to the following sequence of inequalities
\[
(g+1)!\ < \ \Big(\frac{g}{2.5}\Big)^g(g+1) \ = \ \Big(\frac{g+1}{2.5}\Big)^{g+1}\cdot \frac{2.5}{a_g} < \ \Big(\frac{g+1}{2.5}\Big)^{g+1} \ ,
\]
for any $g\geq 6$ and this implies the statement.
\end{proof}

\section{Bounds on the multiplicity for irreducible theta divisors}

The main goal of this section is to prove Theorem~\ref{thm:main1}. We generalize first some work of Debarre \cite{D94} on lower bounds on the degree of non-degenerate curves embedded in a ppav. Combining this with the main results of \cite{Loz20} will lead us to interesting arithmetic properties of the Seshadri constant for theta divisors. Finally, applying together the main ideas of \cite{EL97}, differentiation techniques and intersection theory, lead us to a complete proof of Theorem~\ref{thm:main1}.
\subsection{Degree of curves on abelian varieties.}
Here we generalize some ideas from \cite{D94}. Consequently, this uncovers interesting arithmetic properties for small Seshadri constant. 
\begin{proposition}\label{prop:debarre}
	Let $(A,\Theta)$ be an indecomposable ppav of dimension $g$ and $C\subseteq A$ be an irreducible curve containing the origin. If $g'$ is the dimension of the abelian subvariety generated by $C$, then 
	\[
	(\Theta\cdot C) \ \geq \ \textup{mult}_0(C)+\textup{min}\{g-1, g'\} \ .
	\]
\end{proposition}
\begin{remark}
	It is worth pointing out that when $C$ is non-degenerate then 
	\[
	(\Theta\cdot C) \ \geq \ \textup{mult}_0(C)+g-1 \ .
	\]
This explains some of the phenomena seen in many statements from \cite{D94}.
\end{remark}
\begin{proof}[Proof of Proposition~\ref{prop:debarre}]
For two effective cycles of complementary dimension sitting on an abelian variety, one can associate an endomorphism of the space. So for the pair $(\Theta,C)$ we define
\[
\phi\ \deq \ \phi_{\Theta, C} \ : \ A \ \rightarrow \ A \ .
\]
For any point $x\in A$, for which the intersection $\textup{Supp}(\Theta-x)\cap C$ is zero-dimensional, then $\phi(x)$ is the sum of the points in this intersection taken with the appropriate multiplicities. This can be extended to the whole space since any rational map from an abelian variety is actually a morphism.

Let's fix $p=(\Theta\cdot C)$ and by $q=\mult_0(C)\geq 1$. Our first goal is to show
\[
C\ \nsubseteq \textup{Supp}(\Theta-x) 
\]
 for a general point $x\in\textup{Supp}(\Theta)$. If this does not hold, then continuity yields that $C\subseteq \textup{Supp}(\Theta-x)$ for any $x\in \textup{Supp}(\Theta)$. Consequently, a tangency direction $v\in \textup{T}_0(A)$ of a branch of $C$ is contained in $\textup{T}_0(\Theta-x)$ for any $x\in \textup{Supp}(\Theta)$. But this forces the image  of the Gau\ss  \ map, defined by the irreducible divisor $\Theta$, to be degenerate, i.e. contained in a hyperplane, contradicting \cite[Proposition~4.4.1]{BL04}. 
 
 Taking this into account, our next goal is to study the image $\phi(\Theta)$. As we showed above for a general point $x\in \textup{Supp}(\Theta)$ the intersection of $\Theta-x$ and $C$ is zero-dimensional. Thus 
 \[
 \phi(x) \ = \ k\cdot 0_A\ + \ P_1+\ldots +P_{p-k} \ \in \ \underbrace{C+\ldots +C}_{p-k-\textup{times}} \ \subseteq A'\ \deq  \ <C> \ \subseteq A .
 \] 
for some $k\geq q$ ($k$ depends on the choice of the point). Here $A'$ stands for the $g'$-dimensional abelian subvariety generated by the curve $C$.

Now there is a standard lemma, that whenever one has a dominant rational map $f:X\dashrightarrow Y$, then for any ample divisor $D$ on $X$ we know
\[
\textup{dim}(\overline{(f(D))}) \ = \ \textup{min}\{\textup{dim}(X)-1,\textup{dim}(Y)\}\ .
\]
For a complete proof of this statement one can consult \cite[Lemma~4]{BCL14}.

We apply this statement to $\phi$ and $\Theta$. As $\phi$ is not constant, the dimension formula above, the fact, that the image of a morphism between abelian variety is also abelian, and the minimality property of $A'$ force the following inclusion
\[
\underbrace{C+\ldots +C}_{p-k-\textup{times}}  \ \supseteq \ \textup{closure}(\phi(\Theta))  \ \supseteq \ A' \ ,
\]
whenever $g'<g$ and some $k\geq q$. When $g'=g$, the same force $\phi(\Theta)$ to be an ample divisor on $A$.

Finally, whenever $g'<g$ (or $g'=g$), these ideas yield the following equality of sets 
\[
\underbrace{C+\ldots +C}_{p-k-\textup{times}} \ = \ A' \ ,
\]
for some $k\geq q$ (or $C+\ldots +C$ is an ample divisor in $A$). Counting the dimension, this set-theoretical equality cannot hold, if $p-q \leq g'-1$ (or $p-q\leq g-2$). In particular, this implies the inequality in our main statement and finishes the proof.
\end{proof}
In \cite{Loz20} introduces a conjectural picture for lower bounds of Seshadri constants of polarized abelian varieties. Furthermore, some of these statements are checked to hold in small dimensions. So, applying these ideas together with Proposition~\ref{prop:debarre}, lead to the following corollary:
\begin{corollary}\label{cor:seshadri}
	Let $(A,\Theta)$ be as usual a $g$-dimensional indecomposable ppav with $g\geq 6$. Suppose there exists a curve $C\subseteq A$ passing through the origin such that
	\[
	1 < \epsilon(\Theta)\ =  \ \frac{p}{q} \ \leq \ \frac{g}{g-2} \ ,
	\] 
	where $p=(\Theta\cdot C)$ and $q=\mult_0(C)$. Then $q\geq \frac{5}{2}g-5$ and $p\geq \frac{5g}{2}$.
\end{corollary}
\begin{remark}
	The same type of statement holds when the Seshadri constant $\epsilon(\Theta)$ is not defined by a curve. But for our purposes we need only this simpler statement to prove Theorem~\ref{thm:main1}.
\end{remark}

\begin{proof}
	Let $g'$ be the dimension of the abelian subvariety $A'\subseteq A$,  generated by the curve $C$. Suppose $g'\geq 5$, then applying Proposition~\ref{prop:debarre} and condition $g\geq 6$, we get the following two inequalities
	\[
	\frac{p}{q}\ \leq \ \frac{g}{g-2}\textup{ and } p-q\geq 5 \ . 
	\]
	These two instead imply easily the statement.
	
	It remains the cases when $g'\leq 4$. Note that this forces $g>g'$ and by \cite[Lemma~1]{DH07} the restriction $\Theta|_{A'}$ is not a principal polarization. This will lead to a contradiction, due to the work in \cite{Loz20}. To explain this in more details, we divide the proof in smaller three cases.
	 
	If $g'=1$, then $C$ is an elliptic curve. The equality in the statement implies that $(\Theta\cdot C)=1$. So, applying the main result of \cite{N96} forces $\Theta$ not to be irreducible, leading to a contradiction. 
	
	When $g'=2$, as we pointed out just above we know that$(\Theta|_{A'}^2)\geq 4$. So, by \cite[Corollary~3.3]{Loz20} then either $\epsilon(\Theta|_{A'})\geq 2$ or there exists an elliptic curve $F\subseteq A$ such that $(\Theta\cdot F)=1$. Both cases lead to a contradiction to our initial assumptions or the fact that $\Theta$ is irreducible.
	
If remains the cases $g'=3,4$. If $(A',\Theta|_{A'})$ is indecomposable, then \cite[Proposition~1.2]{Loz20} implies  $\epsilon(\Theta|_{A'})\geq 1.6$, since our pair is not principle. But for $g\geq 6$ this lower bound contradicts the one in the statement. Now, assume $(A',\Theta|_{A'})$ is decomposable. By \cite[Proposition~3.4]{MR15}, the Seshadri constant of $\Theta|_{A'}$, and by assumptions also that of $\Theta$, is computed on one of the components of dimension at most three. Restricting $\Theta$ to this component, by \cite[Lemma~1]{DH07} this Cartier divisor is again not principle. By \cite[Theorem~1.3]{Loz20} we then get either a contradiction or this component is again decomposable, which was already treated above.
\end{proof}

\subsection{Irreducible theta divisors and their multiplicities.}
Here we present the proof of Theorem~\ref{thm:main1}. When the dimension of our abelian ambient space is at most five, then the statement follows from \cite[Theorem~4.6]{C08b}. We present the higher-dimensional case in this subsection.
\begin{theorem}\label{thm:upperbound}
	Let $(A,\Theta)$ be an indecomposable ppav of dimension $g\geq 6$. Then the multiplicity at any point of the divisor $\Theta$ is at most $g-2$.
\end{theorem}

\begin{proof}
	We do the proof by contradiction. Taking into account the main result of \cite{SV96}, we can assume without loss of generality that
	\[
	\mult_0(\Theta)\ = \  g-1 \ .
	\]
In particular, $\mu= \mu(\Theta;0)\geq g-1$ and our goal is to get a contradiction.
	
The basic framework of the proof is similar to the one explained for Proposition~\ref{prop:asymptotic}. First, is the idea of differentiation as explained in \cite[Proposition~4.4]{Loz18}, inspired by \cite{ELN94}. Second, we use the main results \cite{EL97} to construct $g$ divisors $H_1,\ldots , H_{g}$, whose classes are close approximations of $\Theta$. The big difference is that we will obtain these divisors by differentiating $\Theta$.

We start by explaining in more details the idea of differentiation on an abelian variety. If $L$ is a line bundle on $A$, let $\mathscr{D}_L^k$ be the sheaf of differential operators of order $\leq k$ on $L$. Taking $R\in |L|$ to be an effective divisor, then there is a natural surjective homomorphism $\mathscr{D}_L^k\rightarrow L$, that locally assigns to a differential operator $D$ the function $D(f)$, where $f$ is a local function representing the divisor $R$, see \cite[Section~2]{ELN94} for a detailed description of this process.
	
Adjusting \cite[Lemma~2.5]{ELN94} to the pair $(A,\Theta)$ and making use of the classical fact that the divisor $3\Theta$ defines a very ample line bundle, one can show that for any $r,k\geq 1$ the twisted sheaf
	\[
	\mathscr{D}_{r\Theta}^k\otimes \sO_A((3g+3)\Theta)
	\]
	 is globally generated. Let $R=r\Theta$ be our effective divisor, then this defines a linear subsystem 
	 \[
	 V_r^k(\Theta) \ \deq \ \textup{Im}\Big(H^0\big(A,\mathscr{D}_{r\Theta}^k\otimes \sO_A((3g+3)\Theta)\big)\longrightarrow H^0\big(A, \sO_A((r+3g+3)\Theta\big)\Big) \ ,
	 \]
	 whose sections are obtained by process of differentiation of the divisor $r\Theta$ by differential operators of degree $\leq k$. Furthermore, this linear system globally generates the line bundle on the right.
	 
	 Now, for any $i=1, \ldots ,g-2$ and integer $m\geq 1$ we consider the system $V_m^{mi}$. Its elements are obtained by differentiating. So, the fact that $\mult_0(\Theta)\geq g-1$ yields
	 \[
	 H\ \equiv m\cdot (1+\delta)\Theta \textup{ and } \mult_0(H)\geq m\cdot (g-1-i) \ ,
	 \] 
for any $H\in V_m^{mi}$ and setting $\delta=\frac{3g+3}{m}$. 
	
	In the following we let $m\gg 0$, which also translates into $0<\delta\ll 1$. Before going to define the divisors $H_1,\ldots ,H_g$ we need to understand better the behaviour of the following loci
	\[
	A \ \supseteq \ \mathbb{B}^{g-i-1}(\Theta)\ \deq \ \textup{base locus}\Big(V_m^{m(g-i-1)}\Big) \ ,
	\]
	for each $i=1,\ldots ,g-2$. 
	
	Based on the main result of \cite{EL97}, our first goal is to show the following lower bound:
	\begin{equation}\label{eq:base}
	\textup{codim}_A(\mathbb{B}^{g-i-1}(\Theta)) \ \geq \ g-i+1, \textup{ for any } i=1, \ldots ,g-2 \ .
	\end{equation}
	To do so, fix $i=1,\ldots ,g-2$ and let $V\subseteq \mathbb{B}^{g-i-1}(\Theta)$ be a subvariety of codimension at most $g-i$. By above, for any differential operator of order $\leq m(g-i-1)$ there is a divisor in $V_m^{m(g-i-1)}$, locally constructed by applying this operator to $m\Theta$. But $V$ is contained in the support of all the divisors in $V_m^{m(g-i-1)}$, and as the multiplicity is defined locally in terms of differential operators, then
	\[
	\mult_V(m\cdot\Theta) \ \geq  \ m(g-i-1)+1 \ .
	\]
Thus a subvariety of codimension at most $g-i$ is in $\Sigma_{g-i}(\Theta)$, as $\Theta$ is an integral divisor. By \cite[Corollary~2]{EL97} this holds only when $(A,\Theta)$ is a product, contradicting the irreducibility of $\Theta$.
	
In what remains we assume $(\ref{eq:base})$ holds. Our next goal is to construct divisors $H_1,\ldots ,H_{g}$ with
		\begin{equation}\label{eq:intersection}
	\textup{Supp}\Big(H_1\ \cap \ \ldots \ \cap \ H_g \Big) \textup{ is finitely many points.}
	\end{equation}
	For each $i=2,\ldots ,g-1$ consider the following divisor 
	\[
	H_i \ = \ \frac{1}{m}H^i, 
	\]
	where $H^i\in V_m^{m(g-i-1)}$ is a very general choice of divisor in the linear system. In particular, by $H_i\equiv (1+\delta)\Theta$ and $\mult_0(H_i)\geq i$. Moreover, $H_{g-1}=\Theta$, as there is no differentiation here and we choose $H_g=\pi_*(\overline{H}_g)$, where $\overline{H}_g$ is a general choice divisor in the class $\Theta_{\frac{g-2}{1+\delta}}$. 
	
We choose $H_1$ slightly later, but first let's prove inductively that $H_i\cap \ldots \cap H_g$ has the support of codimension $g-i$ for all $i\geq 2$. So, first note that $(\ref{eq:base})$ forces $\mathbb{B}^1(\Theta)$ to be of codimension at least three. Translating this property to the blow-up yields, we deduce then  the following inequality 
	\[
	\textup{codim}_{A}\Big(\pi\big(\Bstable\big(\Theta_{\frac{g-2}{1+\delta}}\big)\big)\Big)\ \geq \ 3 \ .
	\]  
These two conditions on the base loci forces the intersection of $H_{g-2},H_{g-1}$ and $H_g$ to be an effective cycle of codimension $3$. Consequently, applying inductively $(\ref{eq:base})$,  $H_2,\ldots ,H_g$ have a proper intersection as an effective cycle of dimension one. 
	
	Finally, take $H_1$ to be the proper push-forward by the blow-up morphism of a very general choice of a divisor in the class $\Theta_{\epsilon-\delta}$, where $\epsilon=\epsilon(\Theta)$. Since this class is ample on $\overline{A}$, then we easily deduce by above that our choices of our divisors satisfies automatically $(\ref{eq:intersection})$. 
	
Our next step is to use intersection theory. To do so, we know that the choices, we made, forces our divisors to satisfy the following properties:
	\[
	\mult_0(H_i)\geq i, \textup{ for any }i=2,\ldots g-1,\  \mult_0(H_1)=\epsilon-\delta \textup{ and } \mult_0(H_g)= \frac{g-2}{1+\delta}\ .
	\]
	We also know the class in which each divisor $H_i$ lies. As the intersection of these $g$ divisors on $A$ is zero-dimensional, then B\'ezout's theorem yields the following inequality
	\[
\big(H_1\cdot \ldots \cdot H_{g}\big) \ \geq \ (\epsilon-\delta)\cdot \Big(\frac{g-2}{1+\delta}\Big)\cdot (g-1)! \ .
	\]
Letting $m\rightarrow \infty$, and so $\delta\rightarrow 0$, and using the above properties of each $H_i$, this yields
	\[
	g! \ \geq \ \epsilon\cdot (g-2)\cdot (g-1)! \ .
	\]
In particular, if $\epsilon(\Theta;0)>\frac{g}{g-2}$, we get automatically the desired contradiction.
	
Our final step is to deal with the case when the Seshadri constant of $\Theta$ is small, i.e. 
	\[
	\epsilon(\Theta;0)\ \leq \  \frac{g}{g-2} . 
	\]
In this situation note first that $\mathbb{B}^{g-3}(\Theta)$ is at most one dimensional by $(\ref{eq:base})$. As each divisor in $V_{m}^{m(g-3)}$ has multiplicity at least $2$ at the origin, translating this data to the blow-up $\overline{A}$ forces the base locus $\Bstable(\Theta_{2-\delta})$ to be at most one-dimensional for some $0<\delta\ll 1$. So, by Lemma~\ref{lem:lociseshadri} and \cite[Proposition~5.1.9]{PAG} there is a curve $C\subseteq A$, with $p=(\Theta\cdot C)$ and $q=\mult_0(C)$, such that
\[
\epsilon(\Theta;0) \ = \ \frac{p}{q} \ \leq \ \frac{g}{g-2} \ .
\]
Now, going back to the divisors $H_2,\ldots ,H_g$, we defined above, we know that they intersect properly in an effective one-dimensional cycle. So, our next goal is to show that $C$ is contained in this intersection and compute a lower bound on the multiplicity of each divisors along this curve. 

First by \cite[Proposition~4.4]{Loz18} and Lemma~\ref{lem:lociseshadri}, it is not hard to deduce $\mult_{C}(H_{g-1})\ \geq  \ g-2$ as $H_{g-1}=\Theta$ and thus integral. For the rest of $i=2,\ldots ,g-2$ the divisor $H_i$ is obtained  by applying a general choice of a differential operator of order $m(g-i-1)$ to $m\Theta$ and then normalizing. Thus the inequality just above and the behaviour of multiplicity under differentiation imply that
\[
\mult_{C}(H_i) \ \geq \ i-1 \ ,
\]
for all $i=2,\ldots ,g-1$. 

Finally, as $\mult_0(H_g)= \frac{g-2}{1+\delta}$, then \cite[Proposition~4.4]{Loz18} yields
\[
\mult_{C}(H_g) \ \geq \ \frac{g-2}{1+\delta}-\frac{g}{g-2} \ .
\]
Going back to intersection theory, $C$ appears in the support of the one-dimensional effective cycle $H_2\cap\ldots\cap H_g$ with the multiplicity at least the product of the multiplicities of each $H_i$ along $C$. Applying B\'ezout provides then the following inequality
\[
\big(\Theta\cdot H_2\cdot \ldots \cdot H_{g}\big) \ \geq \  (g-2)!\cdot \Big(\frac{g-2}{1+\delta}-\frac{g}{g-2}\Big)\cdot  (\Theta\cdot C) \ .
\]
Letting $m\rightarrow \infty$, and so $\delta\rightarrow 0$, and using the above properties of each $H_i$, this yields
\[
g! \ \geq \ (g-2)!\cdot \Big(g-2-\frac{g}{g-2}\Big)\cdot (\Theta\cdot C) \ .
\]
Finally, our assumption on the upper bound on the Seshadri constant of $\Theta$ allows us to apply Corollary~\ref{cor:seshadri}. Plugging this into the inequality above yields
\[
g! \ \geq \ (g-2)!\cdot (g-2-\frac{g}{g-2})\cdot \frac{5g}{2} \ .
\]
But whenever $g\geq 6$, it easy to see that this inequality cannot hold. So, we get a contradiction also for this final step and finish the proof.
\end{proof}

\section{Bounds on the infinitesimal width for irreducible theta divisors}
In this section we study non-trivial upper bounds on the infinitesimal width of indecomposable ppavs. In particular, we provide a proof to Theorem~\ref{thm:main2} and its stronger version in small dimensions.

For the former, the main ingredient is a criteria, based on \cite{LPP11}, that checks when the linear system of an ample line bundle on an abelian variety separates tangency directions at each point. This is done in terms of the existence of effective $\QQ$-divisors with certain singularities.

Finally, Debarre's conjecture is known for $g\leq 4$. So, using the arithmetic properties of intersection numbers together with plenty of technical computations, we give effective and concrete bounds on the infinitesimal width of a theta divisor sitting on abelian varieties of dimension up to four.

\subsection{Proof of Theorem~\ref{thm:main2}.}

In this subsection a pair $(A,L)$ is an abelian $g$-dimensional variety $A$ and $L$ an ample line bundle on $A$. Inspired by \cite{LPP11}, we provide a criteria for separation of tangency directions by $|L|$  in terms of the existence of divisors with certain singularities.
\begin{proposition}\label{prop:tangency}
	Let $(A,L)$ be a polarized abelian variety of dimension $g$. Suppose that on $A$ there exists an effective $\QQ$-divisor $D\equiv c\dot L$, for some $0<c<1$,  satisfying the following two conditions:
	\begin{enumerate}
		\item $\textup{Supp}(\sJ(A;D))$ is zero-dimensional in a neighborhood of the origin.
		\item  $\sJ(A;D)\subseteq \m_{A,0}^2$.
	\end{enumerate}
	Then $L$ separates tangency directions at each point of $A$.
\end{proposition}

\begin{proof}
	We divide the proof in three steps. 
	
	$\fbox{\textit{Step: 1}}$ Whenever $\sJ(A;D)=\m^2_{A,0}$, the statement holds.
	
   By our assumption, then $\sJ(A;D+x)=\m^2_{A,x}$ for any $x\in A$. So, $H^1(A,L\otimes \m_{A,x}^2)= 0$ for any $x\in A$, by Nadel's vanishing. This forces the surjectiveness of the map on global sections
	\[
	H^0(L\otimes \m_{A,x}) \ \longrightarrow \ H^0\Big(L\otimes \frac{\m_{A,x}}{\m_{A,x}^2} \Big) \ \simeq \ T_xA, 
	\]
implying that the global sections of $L$ separate tangency directions at the point $x\in A$.
	
	$\fbox{\textit{Step: 2}}$ The statement holds if $\sJ(A;D)_{A,0}\subseteq \m^2_{A,0}$ and $\textup{Supp}(\sJ(A;D))$ is zero-dimensional.
	
	The first assumption yields the following short exact sequence:
	\[
	0 \ \rightarrow \ \sJ(A;D+x)\ \rightarrow \ \m^2_{A,x}\ \rightarrow \ \frac{\m^2_{A,x}}{\sJ(A;D+x)}\ \rightarrow \ 0 \ .
	\]
	The second assumption implies that the sheaf on the right has zero-dimensional support. Taking the long exact sequence in cohomology and applying Nadel vanishing, we deduce again that $H^1(A,L\otimes \m^2_{A,x}) = 0$, for any $x\in A$.	The ideas from Step~$1$ lead then again to the desired conclusion.
	
	$\fbox{\textit{Step: 3}}$ Whenever the assumptions in the statement hold, the proposition holds.	
	
	This part is inspired by \cite[Lemma~3.1]{LS12}. The idea is to write our multiplier ideal as a product
	\[
	\sJ(A;D+x)\ = \ \sI(x)  \otimes_{\sO_A}  \sI_{Y_x|A}\ = \ \sI(x) \ \cdot \ \sI_{Y_x|A} \ ,
	\]
	where $Y_x\subseteq A$ a subscheme, with $x\notin Y_x$ and $\sI(x)\subseteq \m^2_{A,x}$ is an ideal with zero-dimensional support. 
	
	Under these assumptions, we then can associate two short exact sequences
	\[
	0\rightarrow \sI(x)\otimes_{\sO_A}\sI_{Y_x|A}\rightarrow \sI(x)\rightarrow \frac{\sI(x)}{\sI(x) \cdot \sI_{Y_x|A}}	\rightarrow 0
	\]
	and 
	\[
	0\rightarrow \sI(x)\otimes_{\sO_A}\sI_{Y_x|A}\rightarrow \sO_{A}\rightarrow \frac{\sO_A}{\sI(x) \cdot \sI_{Y_x|A}}	\rightarrow 0 \ .
	\]
	Since $x\notin Y_x$, the sheaf on the right in the first sequence is a direct summand of the sheaf on the right of the second one. So, tensoring both sequences with $L$, taking the long exact sequences in cohomology, and applying this latter idea together with Nadel's vanishing, will force $H^1(A,\sI(x)\otimes L)=0$ for any $x\in A$. 
Finally, this vanishing together with Step~2 imply easily our statement. 
\end{proof}

With this proposition in hand, we present now the proof of Theorem~\ref{thm:main2}.
\begin{proof}[Proof of Theorem~\ref{thm:main2}]
	Assume that our conclusion does not hold, meaning that
	\[
	\mu(\Theta) \ > \ g-\frac{g-1}{g+1}\ .
	\]
	Our goal is to get a contradiction. 
	
Based on our assumptions we also know $\epsilon(\Theta;0)\geq \frac{2g}{g+1}$. Now, consider a rational number $0<\delta \ll 1$ and two very general choices of $\QQ$-effective divisors 
	\[
	\overline{D}_1\equiv \Theta_{\frac{2g}{2g+1}-\frac{\delta}{2}}\ \textup{ and }\ \overline{D}_2\equiv \Theta_{g-\frac{g-1}{g+1}+\delta} \ .
	\]
	on the blow-up $\overline{A}$ of $A$ at the origin. As usually, we choose these two divisors in the following way. We consider the linear system of some really large power of these rational big classes, so everything makes sense, and then take a general choice in each system,  and finally normalize. The existence of the second divisor is due to our assumption on the infinitesimal width of the theta divisor.
	
	Now consider the push-forward divisors $D_1=\pi_*(\overline{D}_1)$ and $D_2=\pi_*(\overline{D}_2)$. Since $D_2\equiv \Theta$, then by \cite[Proposition~3.5]{EL97} we know that this divisor is log-canonical, i.e. 
	\[
	\sJ(A, (1-c)D_2) \ = \ \sO_A \ , \textup{ for any } 0<c<1 \ .
	\]
	On the other hand, our lower bound on the Seshadri constant yields that the first divisor $\overline{D}_1$ moves in an ample class. So, making use of \cite[Example~9.2.29]{PAG}, we obtain that 
	\[
	\sJ(A;D_1+(1-c)D_2)|_{A\setminus\{0\}} \ = \ \sO_{A\setminus\{0\}}, \textup{ for any } 0<c<1 \ .
	\] 
	In particular, this says that the support of the multiplier ideal is zero-dimensional.
	
	On the other hand, by taking $c\ll \delta$, it is not hard to see that $\mult_0(D_1+(1-c)D_2)\geq g+1$. But this lower bound, by \cite[Proposition~9.3.2]{PAG}, forces the following inclusion
	\[
	\sJ(A;D_1+(1-c)D_2) \ \subseteq \ \m^2_{A,0} \ .
	\]
	With these ideas in hand, then Proposition~\ref{prop:tangency} implies that the linear system associated to the line bundle $L=\sO_A(2\Theta)$ separates the tangency directions at any point on $A$.
	
	But this latter statement clearly cannot hold. The main reason is that when $\Theta$ is irreducible, the linear system $|2\Theta|$ defines a $2:1$ map $A\rightarrow \PP^{2^g-1}$, where the image of this map is the Kummer variety $\textup{Kum}(A)=A/\{\pm\}$. And this map is not \'etale, since it is ramified at the two torsion points. So, there is no way that it can separate all the tangency directions at exactly these $2$-torsion points. This leads to a contradiction and finishes the proof of the theorem.
\end{proof}

\subsection{Infinitesimal width for abelian three-folds}
Here we study the infinitesimal width of an irreducible theta divisor sitting on an abelian three-fold. 
\begin{theorem}\label{thm:abelian3}
	Let $(A,\Theta)$ be an indecomposable three-dimensional ppav. Then 
	\[
	\mult_x(D) \ \leq \ 2 \  ,
	\]
	for any effective $\QQ$-divisor $D\equiv \Theta$ and any $x\in A$. In particular, $\mu(\Theta)\leq 2$.
\end{theorem}
In dimension three the situation is very specific, since any ppav is either the Jacobian of a smooth curve of genus three or is decomposable, see \cite{H63}. So, the proof of Theorem~\ref{thm:abelian3} leads to the following nice consequence:
	\begin{corollary}
		Under the same assumptions, let $D\equiv \Theta$ be an effective $\QQ$-divisor with 
		\[
		\mult_0(D) \ = \ 2 \ .
		\]
		Then one of the two cases take place:
		\begin{enumerate}
			\item If $(A,\Theta)$ is the Jacobian of a hyperelliptic curve, then $D=\Theta$.
			\item 	If $(A,\Theta)\simeq (JC, \Theta_C)$ for a nonhyperelliptic curve $C$, then $D=\Sigma$, where $\Sigma\deq \frac{1}{2}\cdot (C-C)$.
		\end{enumerate}
	\end{corollary}
\begin{proof}[Proof of Theorem~\ref{thm:abelian3}]
	As usual the proof uses intersection theory, as seen before, and the behaviour of the Seshadri constant of the theta divisor. For the latter, \cite{BS01} shows that 
	\[
\epsilon\ \deq \ \epsilon(\Theta) \ = \ \begin{cases} 
1, & (A,\Theta)\textup{ is a polarized product.} \\
\frac{3}{2}, &  (A,\Theta) \textup{ is the Jacobian of a hyperelliptic curve}.\\
\frac{12}{7}, &  (A,\Theta) \textup{ is the Jacobian of a non-hyperelliptic curve}.
\end{cases}	
	\]
An abelian three-fold is always the Jacobian of a curve and this description allows us to assume that $(A,\Theta)$ is such for a smooth complex curve. Moreover, there is a special effective $\QQ$-divisor $D\equiv \Theta$ with an irreducible support and $\mult_0(D)=2$. For hyperelliptic curves \cite[Theorem~2.4]{GG86} yields $D=\Theta$ and for non-hyperelliptic ones  \cite[11.2.8]{BL04} provides $D=\Sigma=\frac{1}{2}\cdot (C-C)$.

With this in hand, we prove the statement by contradiction. So, assume $\mu\deq \mu(\Theta)>2$. Let $\overline{D}\equiv \Theta_2$ be the proper-transform through the blow-up of the origin of $D$. Since $\overline{D}$ is irreducible, then either $\overline{D}\subseteq \Bstable(\Theta_2)$ or $\Bstable(\Theta_2)$ is one-dimensional. As $\mu(\Theta)>2$, the first case cannot happen. 

Now, assume that $\Bstable(\Theta_2)$ is one-dimensional and $\mu(\Theta)>2$. Let $0<\delta\ll 1$ and denote by 
\[
D_{\mu-\delta}=\pi_*(\overline{D}_{\mu-\delta}), \ D_2=\pi_*(\overline{D}_{2}), \textup{ and } D_{\epsilon-\delta}=\pi_*(\overline{D}_{\epsilon-\delta}),
\]
 where $\overline{D}_t\equiv\Theta_t$ is a very general choice of an effective divisor. Note that $\overline{D}_2$ moves in a class with a one-dimensional base locus and $\overline{D}_{\epsilon-\delta}$ in an ample one. Thus, $D_{\mu-\delta}, D_2$ and $D_{\epsilon-\delta}$ intersect properly in a zero-dimensional subscheme, and applying B\'ezout's theorem, we get the inequality:
	\[
	6 \ = \ \big(D_{\epsilon-\delta}\cdot D_2\cdot D_{\mu-\delta}\big) \ \geq \ \mult_0(D_{\epsilon-\delta})\cdot \mult_0(D_2)\cdot \mult_0(D_{\mu-\delta}) \ \geq \ 2(\mu-\delta)\cdot (\epsilon-\delta) \ .
	\]
Letting $\delta\rightarrow 0$ and using the assumption that $\mu>2$, this finally implies $\epsilon(\Theta) <1.5$. By the description of the Seshadri constant, given above, we obtain a contradiction.
\end{proof}

\subsection{Infinitesimal width for abelian four-folds.}
We start with a lemma, known to the experts, but with no good reference we include the proof here.
\begin{lemma} \label{lemma:nef}
	Let $A$ be a $g$-dimensional abelian variety and let $D$ be an effective Cartier divisor on $A$ with $D^g=0$. Then there exists a surjective morphism of abelian varieties $f:A \rightarrow A_1$, with non-zero dimensional fibers, and an effective ample divisor $D'$ on $A_1$ such that $D=f^*(D')$.
\end{lemma}
\begin{proof}
Keeping the notation from \cite{BL04}, let $L=\sO_A(D)$ and $H_L$ its Hermitian metric. Since $H^0(A,L)\neq 0$, $L$ is nef and so $H_L$ has no negative eigenvalues. By \cite[Theorem~3.6.3]{BL04}, we have 
	\[
	\chi(A;L)\ = \ \frac{1}{g!}\cdot (D^g) \ = 0 \ .
	\]
   This instead implies that there is a $q>0$ with $H^q(A,L)\neq 0$. By \cite[Corollary~3.5.4]{BL04}, the existence of such a $q$ forces the restriction $L|_{K(L)_0}$ to be trivial. Moreover, by \cite[Theorem~3.4.5]{BL04}, the number of positive eigenvalues $r$ of the Hermitian metric $H_L$ has to be less than $g$. Thus,
	\[
	\textup{dim}(K(L)_0) \ = \ g-r \ > \ 0 \ .
	\]
	Considering the projection map $f:A\rightarrow A_1\deq A/K(L)_0$ and applying \cite[Lemma~3.3.2]{BL04}, we then easily deduce the statement.
	\end{proof}
With this in hand the main goal of this subsection is to prove the following result:
\begin{theorem}\label{thm:abelian4}
	Let $(A,\Theta)$ be an indecomposable principally polarized abelian four-fold. Then 
	\[
	mult_x(D) \ \leq   \ \begin{cases} 
	 3, &  \ (A,\Theta) \ - \ the \ Jacobian \ of \ a \ hyperelliptic \ curve, \\
	 \frac{11}{4}, & \ (A,\Theta) \ - \ otherwise,
	\end{cases}	
	\]
	for any effective $\QQ$-divisor $D\equiv \Theta$ and any $x\in A$.
\end{theorem}
First we need a technical lemma on the behaviour of codimension one components of the base loci of classes $\Theta_t$ on the blow-up $\overline{A}$. 
\begin{lemma}\label{lem:intersection}
	Let $(A,\Theta)$ be a $4$-dimensional ppav. Let $\overline{R}\geq 0$ be an effective irreducible divisor on $\overline{A}$ so that $\overline{R}\subseteq \Bstable(\Theta_{t_0})$ for some $t_0\in (0,\mu(\Theta))$. Then
	\[
    \mult_{\overline{R}}(||\Theta_t||) \ \leq \ \frac{1}{2}
	\]
	for any $t<\mu(\Theta)$.
	\end{lemma}
\begin{remark}\label{rem:intersection}
	Let $D$ and $F$ be nef Cartier divisors on $A$. Then \cite[Theorem~3.6.3]{BL04} yields that both $D^4$ and $F^4$ are divisible by $24$. Due to this and the binomial extension
	\[
	(D+F)^4 \ = \ D^4+F^4\ + \ 4\big((D^3\cdot F)+(D\cdot F^3)\big)+6(D^2\cdot F^2) \ ,
	\]
	$(D^3\cdot F)+(D\cdot F^3)$ is always divisible by $3$ and $(D^2\cdot F^2)$ is even. Moreover, when $(D^2\cdot F^2)$ is divisible by $4$, then $(D^3\cdot F)+(D\cdot F^3)$ is divisible by $6$.
	\end{remark}
	\begin{proof}
Assume for a fixed $t<\mu(\Theta)$ the opposite inequality hold and the goal is to get a contradiction. Denoting $R\deq \pi_*(\overline{R})$, this yields that $M\deq 2\Theta-R$ is an ample class. 
		
Our first goal is to understand the possible geometry of $R$. More specifically, we want to show
			\begin{equation}\label{eq:three}
			 R \ - \ \textup{ is nef but not ample}; \textup{ and either } \mult_0(R)\geq 2 \textup{ or } R\textup{ is abelian}.
			\end{equation}
Assume first $R$ is ample. Then $24$ divides both $R^4$ and $M^4$. Consequently,   $(M^i\cdot R^{4-i})\geq 24$ by \cite[Theorem~1.6.1]{PAG}. As $M+R=2\Theta$ and $\Theta^4= 24$, the binomial extension implies that $M^4=(M^3\cdot R)=24$ and \cite[Proposition~3]{K66}, forces $M=R=\Theta$. So, $R$ is an irreducible theta divisor on $A$. As Conjecture~\ref{conj:main} holds in dimension four, then $\mult_0(R)\leq 2$. Moreover, $\overline{R}\subseteq \Bstable(\Theta_{t_0})$, and so $\mu(\Theta)\leq 2$. But this is not possible, as \cite[Proposition~1.3.1]{PAG} yields  
\[
		\mu(\Theta) \ \geq  \ \sqrt[4]{\Theta^4}\ = \ \sqrt[4]{24} \ >\ 2 .
		\]
		So, if such an $R$ exists then it must be nef and not ample.
			
	   Assume $\textup{mult}_0(R)=1$. If $R$ is not abelian, then $R_x\deq R-x\neq R$ as effective divisors for any very general point $x\in \textup{Supp}(R)$. Translating this to $\overline{A}$, then $\overline{R}_x\equiv _{\textup{num}}\overline{R}$ with distinct supports, where $\overline{R}_x$ is the proper transform of $R_x$. Since $\overline{R}\subseteq \Bstable(\Theta_{t})$, then  $\overline{R}_x\subseteq \Bstable(\Theta_{t})$, as we can consider these base loci to be numerical invariants by \cite[Lemma~2.3]{Loz18}. Moving $x$ around, these  data forces $\Bstable(\Theta_{t})=X$, which is not possible as $t<\mu(\Theta;0)$. So, $R$ must be abelian, implying that the remaining cases left to tackle are those from $(\ref{eq:three})$.
			
	   The second idea is to show the following upper bounds
			\begin{equation}\label{eq:four}
			23 \ \geq \ \textup{max}\{(\Theta^i\cdot R^{4-i})\ |\ i=1,2,3\}.
			\end{equation}
			We deal first with the case $i=3$. As $M$ is ample, $24$ divides $M^4$ and \cite[Corollary~1.6.3]{PAG} yields $M^j\cdot \Theta^{4-j}\geq 24$. As $\Theta^4=24$ and $M+R=2\Theta$, we automatically get $24\geq (\Theta^3\cdot R)$. Equality does not hold, as otherwise $(\Theta ^4)=(\Theta^3\cdot M)=24$ and then $M\equiv R\equiv \Theta$ by \cite[Proposition~3]{K66}, which contradicts $(\ref{eq:three})$ as $R$ is not ample.
			
			For the cases $i=1,2$, \cite[Example~1.6.4]{PAG} yields the following string of inequalities
			\[
				\big(\Theta^3\cdot R\big)^4 \ \geq \ \big(\Theta^4\big)^2\cdot \big(\Theta^2\cdot R^2\big)^2\ \geq \ 24^2\cdot \big(\Theta^3\cdot R\big)\cdot \big(\Theta\cdot R^3\big) \ . 
			\]
          Combining this with the case $i=3$ clearly imply $(\ref{eq:four})$ for the other two cases.
			
			Going forward, as $R$ is effective but not ample, then $(R^4)=0$ by  \cite[Corollary~1.5.18]{PAG}. In particular,  Lemma~\ref{lemma:nef} yields the existence of a surjective morphism $f:A\rightarrow A_1$ of abelian varieties with non-zero dimensional fibers and an ample divisor $D$ on $A_1$ such that $R=f^*(D)$. With this in hand, we divide the rest of the proof in three cases, based on the dimension of $A_1$.
			
			\fbox{\textit{Case} 1:} $\textup{dim}(A_1)=1$, i.e. $R$ is an abelian three-fold.
			
Applying \cite[Lemma~1]{DH07}, and $(\ref{eq:four})$ we deduce that
 $(\Theta^3\cdot R)=12, 18$. Since $R$ is abelian, rigidity lemma implies the vanishing of $(\Theta^2\cdot R^2), (\Theta\cdot R^3)$, and $R^4$. As $M=2\Theta-R$, this forces $M^4\leq 0$, contradicting the amplitude of $M$.  
			
			\fbox{\textit{Case} 2:} $\textup{dim}(A_1)=2$.
			
			Since $D$ is an ample line bundle on an abelian surface then $D^2=2k$ for some $k\geq 1$. Let's first discount $D^2=2$. As $R$ is not abelian, $\textup{mult}_0(R)\geq 2$ by $(\ref{eq:four})$ and also $\mult_0(D)\geq 2$. On the surface $A_1$, this data forces $\epsilon(D)\leq 1$. So by \cite{N96}, there is an elliptic curve $F\subseteq A_1$ with $(D\cdot F)=1$. And here lies the contradiction, as $D$ and $F$ have different support, B\'ezout's theorem yields
			\[
			1\ = \ (D\cdot F)\ \geq \ \mult_0(D)\cdot \mult_0(F) \geq 2 \ .
			\] 
Now, suppose $D^2= 2k$ for some $k\geq 2$. Consider $S=f^{-1}(0)\subseteq A$ the abelian surface. Then \cite[Lemma~1]{DH07} and asymptotic Riemann-Roch yield $(\Theta^2\cdot S)= 2l$ for some $l\geq 2$. Applying $(\ref{eq:four})$, then the data above forces $(\Theta^2\cdot R^2)=16$. Finally, Hodge index yields
			\[
			(\Theta^3\cdot R)^2 \ \geq (\Theta^2\cdot R^2)\cdot (\Theta^4) \ = \ 24\cdot 16  \ .
			\]
			Together with $(\ref{eq:four})$, this implies that the only cases remaining to deal with is when 
			\[
			(\Theta^3\cdot R)\ =\ 20,21,22,23 \ .
			\]
	Using additionally that $(\Theta\cdot R^3)=0$, we are lead to a contradiction in any of these case, because of the arithmetic nature of intersection numbers on $A$, as explained in Remark~\ref{rem:intersection}.
			
			\fbox{\textit{Case} 3:} $\textup{dim}(A_1)=3$.
			
			First, note that $6$ divides $D^3$, as $D$ is ample on $A_1$. Second, the map $f:A\rightarrow A_1$ has elliptic curves as fibers, whose intersection numbers with $\Theta$ are at least two by \cite[Lemma~1]{DH07}. In particular, these yield $(\Theta\cdot R^3)=6k$, for some $k\geq 2$. On the other hand, applying $(\ref{eq:four})$, we easily deduce that the only cases left to tackle are when  $(\Theta\cdot R^3)=12,18$. 
			
Consider first the case $(\Theta\cdot R^3)=18$. By \cite[Example~1.6.4]{PAG}, we have the inequalities
			\[
						\big(\Theta^3\cdot R\big)^4 \ \geq \ \big(\Theta^4\big)^2\cdot \big(\Theta^2\cdot R^2\big)^2\ \geq \ 24^2\cdot \big(\Theta^3\cdot R\big)\cdot \big(\Theta\cdot R^3\big) \ .
			\]
Consequently, one must have $(\Theta^3\cdot R)\geq 22$. Since $3$ divides $(\Theta\cdot R^3)$, then the same holds for $(\Theta^3\cdot R)$ by Remark~\ref{rem:intersection}. Combining this with $(\ref{eq:four})$ leads us to a contradiction in this case.
			
We are left to tackle the case $(\Theta\cdot R^3)=12$. The same reasoning as in the previous case forces $(\Theta^3\cdot R)=21$. Moving forward,  Hodge index yields the following list of inequalities
			\[
			21^4 \ = \ (\Theta^3\cdot R)^4\ \geq \ (\Theta^4)^2\cdot (\Theta^2\cdot R^2)^2 \ \geq \ 24^2\cdot (\Theta^3\cdot R)\cdot (\Theta\cdot R^3) \ = \ 24^2\cdot 21\cdot 12 \ ,
			\]
providing an upper and a lower bound on $(\Theta^2\cdot R^2)$. Together with Remark~\ref{rem:intersection}, they imply that the only case left is when $(\Theta^2\cdot R^2)=18$. 

As $R^4=0$, we know all the intersection numbers between $\Theta$ and $R$. In particular, as $M=2\Theta-R$ is ample, we can deduce easily the same data between $M$ and $\Theta$
			\[
			(\Theta^3\cdot M)=27, (\Theta^2\cdot M^2)=30,  (\Theta\cdot M^3)=36 \textup{ and } M^4=48 \  .
			\]
			Finally, applying \cite[Example~1.6.4]{PAG}, we get then the inequality
			\[
			900\ = \ (\Theta^2\cdot M^2)^2 \ \geq \ (\Theta^3\cdot M)\cdot (\Theta\cdot M^3)\ = \ 27\cdot 36\ = \ 972 \ ,
			\]
			leading to a contradiction in this third case. This finishes the proof of the lemma.
			
	\end{proof}

\begin{proof}[Proof of Theorem~\ref{thm:abelian4}] 
	We start the proof by introducing some notation. Denote by
	\[
	\epsilon \ \deq \ \epsilon(\Theta)\ \leq \ 	t_3\ \deq \ \textup{min}\{t>0\  | \ \textup{dim}\Big(\pi\big(\Bstable(\Theta_t)\big)\Big)\ \geq  \ 3 \} \ \leq \ \mu\deq  \mu(\Theta).
	\]
	We also consider $\overline{H}\neq E$ to be any three-dimensional irreducible hypersurface, that appears in the base locus $\Bstable(\Theta_t)$ for some $t>t_3$. Applying \cite[Proposition~4.4]{Loz18} in this setup we then get 
	\[
	\mult_{\overline{H}}(||\Theta_t||) \ \geq \ t-t_3 \ , \forall \ t\geq t_3 \ .
	\]
	Applying this together with Lemma~\ref{lem:intersection} to each such $\overline{H}$ yields the following inequality:
	\begin{equation}\label{eq:boundmu}
	\mu(\Theta)\ - \ t_3 \ \leq \ \frac{1}{2} \ . 
	\end{equation}
	With this inequality in hand, we turn our attention to the description of the Seshadri constants of four-dimensional indecomposable ppav given in \cite[Section~4]{D04}. In particular, we have
		\[
	 \epsilon(\Theta) \ = \ \begin{cases} 
	\frac{8}{5}, &  (A,\Theta)\ - \ \textup{ the Jacobian of a hyperelliptic curve,} \\
	\geq 2, &  (A,\Theta)\ - \  \textup{ otherwise}.
	\end{cases}	
	\]
	Based on this description we then divide the proof in two cases.
	
	\fbox{\textit{Case} 1:} ``$\epsilon(\Theta;0)\geq 2\ \Longrightarrow \ \mu(\Theta;0)\leq \frac{11}{4}$''.
	
    Under these circumstances we do the usual trick, use intersection theory and B\'ezout's theorem on $A$. For this fix a positive number $0<\delta\ll 1$ and consider very general choices of divisors 
	\[
	\overline{D}_1,\overline{D}_2 \ \equiv \ \Theta_{2-\delta}, \overline{D}_3\ \equiv \ \Theta_{t_3-\delta}, \textup{ and } \overline{D}_4 \ \equiv \ \Theta_{\mu-\delta} \ .
	\]
	The first two divisors sit in an ample class, as $\epsilon(\Theta;0)\geq 2$. The divisor $\overline{D}_3$ lives in a big class with a base locus of dimension at most two, even those contained in thee exceptional divisor $E\simeq\PP^2$, and $\overline{D}_4$  of dimension at most three. Setting $D_i=\pi_*(\overline{D}_i)$, then the intersection of all the divisors $D_i$ is zero-dimensional. Hence, B\'ezout's theorem yields the following inequality:
	\[
	24 \ =\ (\Theta^4) \ = \ (D_1\cdot D_2\cdot D_3\cdot D_4)\ \geq \ \prod_{i=1}^{i=4}\textup{mult}_0(D_i)\ = \ (2-\delta)^2\cdot (t_3-\delta)\cdot(\mu-\delta) \ .
	\]
Taking $\delta\rightarrow 0$, and considering $(\ref{eq:boundmu})$, we are lead to the following one
\[
24 \ \geq \ 4\mu(\mu-0.5) \ .
\]
But this does not hold for $\mu\geq  2.75$, and finishes the proof in this case.

 	\fbox{\textit{Case} 2:} ``$\epsilon(\theta;0)< 2\ \Longrightarrow \ \mu(\Theta;0)\leq 3$''.

In this case the idea is to do the intersection theory part on $\overline{A}$. This could complicate matters, because of the irreducible components appearing in our base loci, contained in $E$, but at least in dimension four this seems not to be the case. 

Based on the description of the Seshadri constant above we know that in this case the pair $(A,\Theta)$ is the Jacobian of a hyperelliptic curve $C$ and $\epsilon(\Theta)=\frac{8}{5}$. So, fixing a rational number $0<\delta\ll 1$, we consider the following intersection number
\[
 (\Theta_{\frac{8}{5}-\delta}\cdot \Theta_{2-\delta}\cdot \Theta_{t_3-\delta}\cdot \Theta_{\mu-\delta}) \ = \ 24-(\frac{8}{5}-\delta)(2-\delta)(t_3-\delta)(\mu-\delta) \ .
\]
The idea is to show geometrically that this intersection number is positive. If we do so, then taking $\delta\rightarrow 0$ and using $(\ref{eq:boundmu})$, will force automatically $\mu(\Theta)\leq 3$ and finish the proof. 

In order to understand the intersection number above, note that a very general choice of divisor in $\Theta_{t_3-\delta}$ and one in $\Theta_{\mu-\delta}$ would intersect properly in an effective cycle of codimension two. This follows from the definition of $t_3$ and the fact that any irreducible components of $\Bstable(\Theta_{t_3-\delta})$ cannot be contained in $E$ and be three-dimensional.

As a consequence of these ideas, we can then prove that the intersection number above is positive, as long as we are able to show 
\begin{equation}\label{eq:surface}
(\Theta_{\frac{8}{5}-\delta}\cdot \Theta_{2-\delta}\cdot \overline{S}) \ \geq \ 0 \ ,
\end{equation}
for any surface $\overline{S}\subseteq \overline{A}$.

In order to do so, we will take advantage of some classical aspects of the geometry of the Jacobian of a hyperelliptic curve. First, the curve $C$ is canonically embedded in $A$, and we can define the surface $S_C=C-C$. Applying simultaneously \cite[Theorem~2.4]{GG86} and \cite[Theorem~1]{BD88} to our infinitesimal setup on the blow-up $\overline{A}$, yields the inclusion
\[
\Bstable(\Theta_2) \ \subseteq \ \overline{S}_C \ \subseteq \ \overline{A} \ ,
\]
where $\overline{S}_C$ is the proper transform of $S_C$. 

With this in hand, we proceed to prove $(\ref{eq:surface})$. If $\overline{S}=\overline{S}_C$, then \cite[p. 79]{F84} yields
\[
\mult_0(S_C)\ = \ -(\overline{S}_C\cdot (-E)^2) \ .
\]
This equality, projection formula for intersection numbers, and the facts that $(\Theta^2\cdot S_C)=12$ and $\mult_0(S_C)=3$, imply easily $(\ref{eq:surface})$ in this case.

If $\overline{S}\neq \overline{S}_C$, then $\overline{S}\nsubseteq \Bstable(\Theta_{2-\delta})$ by above. Thus a general choice of a divisor in $\Theta_{2-\delta}$ will intersect $\overline{S}$ in a one-dimensional effective cycle on $\overline{A}$. As the class $\Theta_{\frac{8}{5}-\delta}$ is ample, then the desired positivity statement in $(\ref{eq:surface})$ follows and we finish the proof.
	\end{proof}

\end{document}